\documentclass[reqno,a4paper,10pt]{amsart}
\usepackage{amsthm,amsmath,amsfonts}
\usepackage{enumerate}
\usepackage{graphicx}
\usepackage{hyperref}
\theoremstyle{definition}
\newtheorem{definition}{Definition}
\newtheorem{example}[definition]{Example}
\theoremstyle{remark}
\newtheorem{remark}[definition]{Remark}
\theoremstyle{plain}
\newtheorem{lemma}[definition]{Lemma}

\newtheorem{theorem}[definition]{Theorem}

\newcommand{\set}[1]{\left\{{#1}\right\}}
\newcommand{\vek}[1]{\boldsymbol{#1}}

\newcommand\setsuchas[2]{\left\{\,{#1}\,:\,{#2}\,\right\}}
\newcommand{\tdeg}[1]{\lvert #1 \rvert}
\newcommand{\Nat}{{\mathbb{N}}}
\newcommand{\Z}{{\mathbb{Z}}}

\newcommand{\Q}{{\mathbb{Q}}}
\newcommand{\divides}[2]{{#1} \lvert {#2}}
\newcommand{\ind}{\mathrm{ind}}
\newcommand{\Ind}{\mathrm{Ind}}

\begin{document}
\title{Generating functions for borders}
\author{Jan Snellman}
\address{Department of Mathematics, Link\"oping University\\
SE 58183 Link\"oping, Sweden}
\email{jan.snellman@liu.se}

\begin{abstract}
  We give the generating function for the ``index'' function on integer lattice
  points, relative to a finite order ideal. The index is an important
  concept in the theory of border bases, an alternative to Gr\"obner
  bases. 
\end{abstract}

\keywords{Border bases, monomial ideals, generating
functions}
\subjclass{05A15; 05A17}
\maketitle

\begin{section}{Introduction}
  Let \(S=K[x_{1},\dots,x_{n}]\) be the polynomial ring over a field \(K\), and let
  \(I \subset S\) be an ideal with Krull-dimension zero; i.e., the quotient ring \(S/I\)
  has finite dimension as a \(K\)-vector space. Suppose further that
  \[\mathcal{O}= \{t_{1},\dots, t_{\mu}\} \subset T([n]) = \setsuchas{x^\alpha}{\alpha \in   \Nat^n}\]
  is an on \emph{order ideal} with
  respect to the divisibility ordering of \(T([n])\). In other words,
  \(\mathcal{O}\) is a set of power products which is closed under taking divisors.
  Then an \(\mathcal{O}\)-border basis
  \cite{KehKreu:CBB,
    kehrein2005algebraist,
    kehrein2006computing,
    kreuzer2008deformations,
    kreuzer2011geometry,
    mourrain2007pythagore,
    mourrain2012border}
  of \(I\) is a set of polynomials \(G=\{g_{1},\dots,g_{\nu}\}\) of the form
  \(g_{j} = b_{j} - \sum_{i=1}^{\mu}c_{ji}t_{i}\), where \(\{b_{1},\dots,b_{\nu}\}\) is the
  \emph{border}
  \[
    \partial \mathcal{O} = (x_{1} \mathcal{O} \cup \cdots \cup x_{n} \mathcal{O}) \setminus \mathcal{O}
  \]
  of \(\mathcal{O}\), and \(c_{ji} \in K\). The \(c_{ji} \in K\) should be such that \(I\) is generated by
  \(G\), and \(\mathcal{O}\) should be a basis for the \(K\)-vector space \(S/I\).

  Border bases are a more recent notion that that of \emph{Gröbner bases}
  \cite{buchberger2006bruno,buchberger1998grobner,froberg1997introduction}
  which also yields a generating set of \(I\) and a set of monomials whcih form
  a \(K\)-vector space basis of \(S/I\).
  Border bases have some advantages:
  \begin{itemize}
    \item Border bases do not depend on a \emph{term order} (a multiplicative total order on \(T([n])\)).
    \item Border bases are more stable with respect to perturbations of the coefficients of the generators of
          \(I\) \cite{abbott2008stable}.
    \item Border bases may be more suitable for ssolving polynomial systems with a high degree of symmetry.
  \end{itemize}

  Of fundamental importance in the theory is the notion of the \emph{index} function
  \(\ind_{\mathcal{O}}(t)\) which measures ``how far away'' a power product \(t\) is from the
  order ideal \(\mathcal{O}\). In two variables, the index is defined by the difference equation
    \begin{equation}
      \label{eq:rec2}
      \ind_{\mathcal{O}}(x_{1}^{a_{1}}x_{2}^{a_{2}}) =
        \begin{cases}
          0 & x_{1}^{a_{1}}x_{2}^{a_{2}} \in \mathcal{O}, \\
          1 + \ind_{\mathcal{O}}(x_{1}^{a_{1}-1}) & x_{1}^{a_{1}}x_{2}^{a_{2}} \not \in \mathcal{O} \text{ and } a_{2}=0, \\
          1 + \ind_{\mathcal{O}}(x_{2}^{a_{2}-1}) & x_{1}^{a_{1}}x_{2}^{a_{2}} \not \in \mathcal{O} \text{ and } a_{1}=0, \\
          1+ \min  \bigl(\ind_{\mathcal{O}}(x_{1}^{a_{1}-1}x_{2}^{a_{2}}) ,
          \ind_{\mathcal{O}}(x_{1}^{a_{1}}x_{2}^{a_{2}-1}) \bigr) & \text{ otherwise. }
      \end{cases}
    \end{equation}

    As an example, let \(\mathcal{O}=\set{1,x,x^{2},y}\).
 Then the index is the integer-valued function
    on \(T([2])\) given by
    \begin{displaymath}
      \left(\begin{array}{rrrrrrr}
5 & 6 & 6 & 7 & 8 & 9 & 10 \\
4 & 5 & 5 & 6 & 7 & 8 & 9 \\
3 & 4 & 4 & 5 & 6 & 7 & 8 \\
2 & 3 & 3 & 4 & 5 & 6 & 7 \\
1 & 2 & 2 & 3 & 4 & 5 & 6 \\
0 & 1 & 1 & 2 & 3 & 4 & 5 \\
0 & 0 & 0 & 1 & 2 & 3 & 4
\end{array}\right),
    \end{displaymath}
    which can be visualized as in Figure~\ref{fig:M1}.
    \begin{figure}[htb]
      \includegraphics[width=0.3\textwidth]{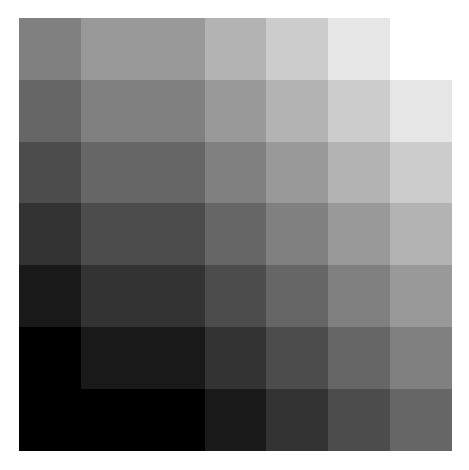}
      \caption{The index function of the order ideal \(\set{1,x,x^{2},y}\). }
      \label{fig:M1}
    \end{figure}

    We can:
    \begin{enumerate}
      \item Regard the index function
            \(x_{1}^{a_{1}}x_{2}^{a_{2}} \mapsto \ind_{\mathcal{O}}(x_{1}^{a_{1}}x_{2}^{a_{2}})\)
            as a function \(c: \Nat^{2} \to \Nat\).
            Then this \(c\) is the unique solution to the difference equation
            \begin{equation}
      \label{eq:rec3}
      c_{i,j} =
        \begin{cases}
          0 & (i,j) \in \mathcal{O}, \\
          1 + c_{i-1,j} & (i,j) \not \in \mathcal{O}, \, i>0, j=0 \\
          1 + c_{i,j-1} & (i,j) \not \in \mathcal{O}, \, i=0, j>0 \\
          1 + \min(c_{i-1,j}, c_{i,j-1}) &  (i,j) \not \in \mathcal{O}, \, i,j > 0
        \end{cases}
    \end{equation}

      \item Form the generating function\footnote{For this example, we switch to new variables; in
            the main text we abuse notation and utilize \(x_{1},\dots,x_{n}\) for two different
            purposes.}
            \(f(u_{1},u_{2})=\sum_{i,j} c_{{i,j}}u_{1}^{i}u_{2}^{j} \in \Z[[u_{i},u_{j}]]\),
    \item
      Express the index as a very special rational function.
    \end{enumerate}
    In the example,
    \begin{displaymath}
      f(u_{1},u_{2}) = \frac{u_1^{2}}{(1-u_{1})(1-u_{2})}
      + \frac{u_{1} + u_{2}}{1-u_{2}}
    \end{displaymath}

    This article is concerned with finding an explicit formula for the generating function
    \(f\) of the index. While certainly inspired by the theory of border bases, we treat this as a
    purely combinatorial problem, where the difference equation \eqref{eq:rec3} is ``almost linear''
    in that it is the minimum of two (in \(n\) variables, \(n\)) (affine) linear forms.
    The study of systems of linear difference equations in several variables are surprisingly deep
    \cite{bousquet2006polynomial,
      bousquet2005walks,
      bousquet2003walks,
      bousquet2010families,
      prea1997exterior,
      flajolet2009analytic}.
    Since the difference equation~\eqref{eq:rec3} involves a right-hand side which is the minimum
    of linear forms, it can be regarded as a ``tropical polynomial''
    \cite{speyer2009tropical,mikhalkin2009tropical}.

    There are a few works that treat tropical difference (and differential) equations
    \cite{grigoriev2017tropical,giansiracusa2024general}; it is a somewhat complicated and technical
    area of study. As a contrast, the solutions to the difference equations involving the index function
    are, while not completely trivial, still very easy and concrete.

\end{section}

\begin{section}{Definition and elementary properties of borders}
  \begin{definition}
    Let \([n]=\set{1,2,\dots,n}\).
    Let \(X([n])=\set{x_1,\dots,x_n}\) and let \(T([n])\) be the
    monoid of power products 
    \[T([n]) = \setsuchas{x^\alpha}{\alpha \in   \Nat^n}.\] If 
    \(S \subseteq [n]\) then 
    \[T(S) =
    \setsuchas{x^\alpha}{\alpha \in \Nat^n, \, i \not \in S \implies
      \alpha_i=0}.\] 
    We denote the degree of a power product \(t=x^\alpha\) by
    \(\tdeg{t} = \tdeg{\alpha}\), and let
    \begin{displaymath}
      \begin{split}
        T(S)_a &= \setsuchas{t \in T(S)}{\tdeg{t}=a} \\
        T(S)_{< a} &= \setsuchas{t \in T(S)}{\tdeg{t} < a} 
    \end{split}
    \end{displaymath}
  \end{definition}
  
  \begin{definition}
    An order ideal \(\mathcal{O} \subset T([n])\) is a set
    of power products closed under division, i.e.
    \begin{displaymath}
      t \in \mathcal{O}, \, \divides{s}{t} \, \implies s \in \mathcal{O}
    \end{displaymath}
    The border of an order ideal is defined by
    \begin{displaymath}
      \partial \mathcal{O} = \left[ T([n])_1 \mathcal{O}\right] \,
      \setminus \, \mathcal{O}
    \end{displaymath}
    and the closed border by
    \begin{displaymath}
      \overline{\partial \mathcal{O}} = \partial \mathcal{O} \, \cup
      \mathcal{O} 
    \end{displaymath}
    By iteration, we get the higher borders
    \begin{displaymath}
      \partial^{k} \mathcal{O} = \partial \left(
        \overline{\partial^{k-1} \mathcal{O}} \right), \qquad
       \overline{\partial^{k} \mathcal{O}} = \partial^{k} \mathcal{O}
       \cup \overline{\partial^{k-1} \mathcal{O}}
    \end{displaymath}
    We also define
    \begin{displaymath}
      \partial^{0} \mathcal{O} = \mathcal{O}, \qquad
      \overline{\partial^{0} \mathcal{O}} = \mathcal{O}
    \end{displaymath}
  \end{definition}

Kehrein and Kreuzer \cite{KehKreu:CBB} showed:
  \begin{lemma}
    Let \(\mathcal{O}\) be an order ideal.
    \begin{enumerate}[(a)]
    \item For every \(k \ge 1\), we have that 
      \[\partial^k \mathcal{O}
      = T([n])_k \cdot \mathcal{O} \setminus \left( T([n])_{<k} \cdot
        \mathcal{O} \right).\]
    \item For every \(k \ge 1\), we have a disjoint union
      \begin{displaymath}
        \overline{\partial^k \mathcal{O}} = \bigcup_{i=0}^k \partial^i
        \mathcal{O}. 
      \end{displaymath}
      In particular, we have a disjoint union
      \begin{displaymath}
        T([n]) = \bigcup_{i=0}^\infty \partial^i \mathcal{O}.
      \end{displaymath}
    \item A term \(t \in T([n])\) is divisible by a term in
      \(\mathcal{O}\) iff \(t \in T([n]) \setminus \mathcal{O}\).
    \end{enumerate}
  \end{lemma}

  \begin{definition}
    Given an order ideal \(\mathcal{O}\) and a power product \(t \in
    T([n])\), we define the index \(\ind_{\mathcal{O}}(t)\) as
    the minimum \(k\) such that \(t \in \overline{\partial^{k}
      \mathcal{O}}\). 
  \end{definition}
  
  Kehrein and Kreuzer \cite{KehKreu:CBB} proved the following properties of the index:
  \begin{lemma}\label{lemma:Kehrein}
    \begin{enumerate}[(a)]
    \item The index \(\ind_{\mathcal{O}}(t)\) is the smallest natural
      number \(k\) such that \(t=t_1t_2\) with \(t_1 \in
      \mathcal{O}\), \(t_2 \in T([n])_k\).
    \item \(\ind(tt') \le \tdeg{t} + \ind_{\mathcal{O}}(t')\).
    \end{enumerate}
  \end{lemma}
  
  We can also note that
  \begin{lemma}
    \begin{equation}
      \label{eq:rec}
      \ind_{\mathcal{O}}(t) =
        \begin{cases}
          0 & t \in \mathcal{O} \\
          \min ( 
          \setsuchas{1+\ind_{\mathcal{O}}(\frac{t}{x_i})}{\divides{x_i}{t}}
          )
          &
          t \not \in \mathcal{O} 
        \end{cases}
    \end{equation}
  \end{lemma}

  \begin{example}\label{ex:311}
   Let \(\mathcal{O}=\set{1,x_1,x_1^2,x_2,x_2^2}\). 
    Then the function \(\ind_{\mathcal{O}}\) looks as follows:
    \begin{displaymath}
      \begin{matrix}
        5 & 6 & 7 & 8 & 9 & 10 & 11 & 12 \\ 
        4 & 5 & 6 & 7 & 8 & 9 & 10 & 11  \\
        3 & 4 & 5 & 6 & 7 & 8 & 9 & 10   \\
        2 & 3 & 4 & 5 & 6 & 7 & 8 & 9    \\
        1 & 2 & 3 & 4 & 5 & 6 & 7 & 8    \\
        0 & 1 & 2 & 3 & 4 & 5 & 6 & 7    \\
        0 & \vek{1} & 1 & 2 & 3 & 4 & 5 & 6    \\
        0 & 0 & 0 & 1 & 2 & 3 & 4 & 5
      \end{matrix}
    \end{displaymath}
  \end{example}

  Using the natural identification
  \begin{displaymath}
    \begin{split}
      T([n]) & \to \Nat^n \\
      \vek{x}^{\vek{\alpha}} & \mapsto \vek{\alpha}
    \end{split}
  \end{displaymath}
  we will, when convenient, regard the index as the function
  \begin{displaymath}
    \ind_{\mathcal{O}}: \Nat^n \to \Nat
  \end{displaymath}
  which is the unique solution to 
  \begin{equation}
  \label{eq:diffeq}
  \ind_{\mathcal{O}}(a_1,\dots,a_n) =
  \begin{cases}
    0 & (a_1,\dots,a_n) \in \mathcal{O}, \\
    \displaystyle \min_{1 \le i \le n}
    \setsuchas{1+\ind_{\mathcal{O}}(a_1,\dots,a_i-1,\dots,a_n)}{a_i>0} & (a_1,\dots,a_n) \not \in \mathcal{O}.
  \end{cases}
    \end{equation}
\end{section}

\begin{section}{Generating functions for the index}

  We want to calculate the generating function for the index of a finite order ideal.
  \begin{definition}
    If \(\mathcal{O} \subset T([n])\) is a finite order ideal, then we
    define
    \begin{equation}
      \label{eq:6}
      \Ind_{\mathcal{O}}(y_1,\dots,y_n) = \sum_{\alpha \in \Nat^n}
      \ind_\mathcal{O}(x^\alpha) y^\alpha
    \end{equation}
  \end{definition}

\begin{subsection}{Dimension 2}  
  If \(n=2\), a finite order ideal \(\mathcal{O}\) may be encoded by a
  (number) partition 
  \(\lambda=(\lambda_1,\dots, \lambda_m)\), as
  \begin{equation}
    \label{eq:2dim}
    t=x_1^{\alpha_1}x_2^{\alpha_2} \in \mathcal{O} \qquad \iff \qquad 
    \alpha_1 <m \text{ and } \alpha_2 < \lambda_{\alpha_1+1}
  \end{equation}
  The minimal axis-parallel rectangle containing \(\mathcal{O}\) is
  called the \emph{bounding rectangle} of \(\mathcal{O}\). Clearly,
  the upper right corner of the bounding rectangle has coordinates
  \((m-1,\lambda_1 -1)\).
  
  \begin{lemma}\label{lemma:increase2}
    Let \(\mathcal{O}\) be a finite order ideal in \(T([2])\), given
    by a partition \((\lambda_1,\dots,\lambda_m)\) as in
    \eqref{eq:2dim}. Then 
    \begin{enumerate}
    \item     If \(i > m-1\) then \(\ind_\mathcal{O}(i+k,j)=
      \ind_\mathcal{O}(i,j)+k\).  
    \item     If \(j > \lambda_1-1\) then \(\ind_\mathcal{O}(i,j+k)=
      \ind_\mathcal{O}(i,j)+k\). 
    \item     If  \(i > m-1\) and \(j > \lambda_1-1\) then then
    \(\ind_\mathcal{O}(i+k,j+\ell)=
    \ind_\mathcal{O}(i,j)+k+\ell\).
    \end{enumerate}
  \end{lemma}
  \begin{proof}
    We prove the first assertion by induction over \(j\). If \(j=0\)
    then it follows from \eqref{eq:rec} that
    \[\ind_{\mathcal{O}}(i+1,0) = 1 + \ind_{\mathcal{O}}(i,0),\] hence
    that 
    \[\ind_\mathcal{O}(i+k,0)= \ind_\mathcal{O}(i,0)\] for all
    positive \(k\).
    Assume that the assertion holds for \(j<p\). Then for \(j=p\) we
    have that 
    \[\ind_{\mathcal{O}}(i+1,j) = 1 + \min(\ind_{\mathcal{O}}(i,j),
    \ind_{\mathcal{O}}(i+1,j-1)).\]
    Thus, either (or both) of the conditions 
    \[
    \ind_{\mathcal{O}}(i+1,j) = 1+ \ind_{\mathcal{O}}(i,j), \quad
    \ind_{\mathcal{O}}(i+1,j) = 1+ \ind_{\mathcal{O}}(i+1,j-1)
    \]
    hold. If the first condition holds, we are through. If
    \[
    \ind_{\mathcal{O}}(i+1,j) = 1+ \ind_{\mathcal{O}}(i+1,j-1)
    \]
    then by the induction hypothesis
    \[
    \ind_{\mathcal{O}}(i+1,j) = 1+ \ind_{\mathcal{O}}(i+1,j-1) = 1 +
    1+ \ind_{\mathcal{O}}(i,j-1)
    \]
    hence 
    Lemma~\ref{lemma:Kehrein} gives that 
    that 
    \[\ind_{\mathcal{O}}(i+1,j) = 1+ \ind_{\mathcal{O}}(i,j) \]

    The second assertion is proved in exactly the same way, and the
    third assertion follows from the first two.
  \end{proof}

  \begin{remark}\label{rem:man}
    The following figure illustrates the manipulations in the proof:
    \begin{displaymath}
      \begin{matrix}
        b & c \\
        a & a+1
      \end{matrix}
    \end{displaymath}
    We know that \(c=1+\min(b,a+1)\) and that \(c \le a+2\). It
    follows that \(c=b+1\).
  \end{remark}

  \begin{theorem}\label{thm:IndDim2}
    Let \(\mathcal{O} \subset T([2])\) be a finite order ideal given by
    a partition 
    \(\lambda=(\lambda_1,\dots,\lambda_m)\) as in \eqref{eq:2dim}.
    Put
    \begin{equation}
      \label{eq:P12}
      \begin{split}
        P_0(;;b) &= b \\
        P_1(y_1;1;b) &= \sum_{i_1=0}^\infty (i_1 +b)y_1^{i_1} =
        \frac{b-(b-1)y_1}{(1-y_1)^2} \\  
        P_2(y_1,y_2;1,1;b) &= \sum_{i_1=0}^\infty
        \sum_{i_2=0}^\infty (i_1 +i_2 b)y_1^{i_1}y_2^{i_2} \\
        &=
        \frac{b-(b-1)(y_1+y_2)+(b-2)(y_1y_2)}{(1-y_1)^2(1-y_2)^2} 
      \end{split}
    \end{equation}
    Then
    \begin{equation}
      \label{eq:Ind2}
      \begin{split}
      \Ind_{\mathcal{O}}(y_1,y_2) 
      & = y_1^my_2^{\lambda_1}
      P_2(y_1,y_2;1,1;\ind_{\mathcal{O}}(m,\lambda_1)) \,+ \\
      & 
      \sum_{j=0}^{m-1}
      y_1^jy_2^{\lambda_1}P_1(y_2;1;\ind_{\mathcal{O}}(j,\lambda_1)) \,+ \\
      & \sum_{i=0}^{\lambda_1-1} y_1^m y_2^j
      P_1(y_1;1;\ind_{\mathcal{O}}(i,\lambda_1-1)) \, +\\
      & \sum_{\substack{x_1^{a_1}x_2^{a_2} \in T([2]) \setminus \mathcal{O}\\
        \, a_1 <m \\ a_2 < \lambda_1}} y_1^{a_1} y_2^{a_2}
      P_0(;;\ind_{\mathcal{O}}(a_1,a_2)) 
      \end{split}
    \end{equation}
    In particular, it is a rational function, with a denominator which
    divides \((1-y_1)^2(1-y_2)^2\).
  \end{theorem}
  \begin{proof}
    Since points in \(T([2]) \setminus \mathcal{O}\) can be partitioned into
    \begin{enumerate}
    \item Points to the right and above \(B\),
    \item Points above, but not to the right of, \(B\),
    \item Points to the right of, but not above, \(B\), and
    \item Points inside \(B\) but not in \(\mathcal{O}\),
    \end{enumerate}
    the identity \eqref{eq:Ind2} follows from
    Lemma~\ref{lemma:increase2}.
  \end{proof}

  If \(\mathcal{O}\) is empty, then \(\ind(a,b)=a+b-1\), which is 
  id:A002024 in \cite{OEIS},
  and the
  generating function is \(P_2(y_1,y_2;1,1;1)\).

\end{subsection}
\begin{subsection}{Higher dimensions}

  \begin{definition}\label{def:P}
    Let \(n\) be a positive integer and let \(a_1,\dots,a_n,b\) be
    rational numbers. Define 
    \begin{equation}
      \label{eq:P}
      P_n(x_1,\dots,x_n;a_1,\dots,a_n;b) = \sum_{\alpha \in \Nat^n}
      (a_1\alpha_1 +\dots + a_n \alpha_n + b) x^\alpha 
    \end{equation}
    We define \(P_0(;;b) = b\).
  \end{definition}

  \begin{theorem}\label{thm:Pn}
    Let \(n\) be a positive integer and let \(a_1,\dots,a_n,b\) be
    rational numbers. Then
    the following identity of 
    formal power series in \(\Q[[x_1,\dots,x_n]]\) holds:
    \begin{equation}
      \label{eq:2}
      P_n(x_1,\dots,x_n;a_1,\dots,a_n;b) = 
      \prod_{i=1}^n(1-x_i)^{-2} \sum_{A \subseteq [n]} \Bigl[
      (-1)^{\tdeg{A}}(b-\sum_{j \in A} a_j) \prod_{j \in A} x_j \Bigr]
    \end{equation}
    In particular, putting \(a_1=a_2 = \cdots = a_n = 1\) we have that
    \begin{equation}
      \label{eq:1}
      \begin{split}
      P_n(x_1,\dots,x_n;1,\dots,1;b) &=
      \sum_{\alpha \in \Nat^n} (\tdeg{\alpha} + b) x^\alpha \\
      &=
      \prod_{i=1}^n(1-x_i)^{-2} \sum_{A \subseteq [n]} \Bigl[
      (-1)^{\tdeg{A}}(b-\tdeg{A}) \prod_{j \in A} x_j \Bigr]
      \end{split}
    \end{equation}
  \end{theorem}
  \begin{proof}
    Put
    \begin{displaymath}
      Q_n(x_1,\dots,x_n;a_1,\dots,a_n;b) =
      \prod_{i=1}^n(1-x_i)^{-2} \sum_{A \subseteq [n]} \Bigl[  
      (-1)^{\tdeg{A}}(b-\sum_{j \in A} a_j) \prod_{j \in A} x_j \Bigr] 
    \end{displaymath}
    Then
    \begin{displaymath}
      \begin{split}
        P_1(x_1;a_1,b) &= \sum_{\alpha_1=0}^\infty
        (a_1\alpha_1 b) x_1^{\alpha_1} \\
        &= a_1 x_1(1-x_1)^{-2} + b(1-x_1)^{-1} \\
        &= (b-(b-a_1)x_1)(1-x_1)^{-2}
      \end{split}
    \end{displaymath}
    Now assume, by induction, that \(P_{n-1}=Q_{n-1}\). 
    Let 
    \begin{displaymath}
      \begin{split}
        \vek{\alpha} & =(\alpha_1,\dots,\alpha_n) \\
        \widetilde{\vek{\alpha}} &= (\alpha_1,\dots,\alpha_{n-1}) \\
        \vek{a} & = (a_1,\dots,a_n) \\
        \widetilde{\vek{a}} &= (a_1,\dots,a_{n-1})\\
        \vek{x} & = (x_1,\dots,x_n) \\
        \widetilde{\vek{x}} &= (x_1,\dots,x_{n-1})
      \end{split}
    \end{displaymath}
    and let \(\bullet\) denote the ordinary scalar product.
    Then
    \begin{displaymath}
      \begin{split}
        P_n &= P_n(\vek{x};\vek{a};b) \\
        &= \sum_{\vek{\alpha} \in \Nat^n}
        (\vek{a} \bullet \vek{\alpha} +b) \vek{x}^{\vek{\alpha}} \\
        &= \sum_{\alpha_n=0}^\infty \sum_{\overline{\vek{\alpha}} \in
          \Nat^{n-1}}
        (\overline{\vek{a}} \bullet \overline{\vek{\alpha}} + a_n\alpha_n
        +b) \overline{\vek{x}}^{\overline{\vek{\alpha}}} x_n^{\alpha_n} \\
        &=\sum_{\alpha_n=0}^\infty x_n^{\alpha_n}
        \sum_{\overline{\vek{\alpha}} \in \Nat^{n-1}}
        P_{n-1}(\overline{\vek{x}}; \overline{\vek{a}};a_n \alpha_n +b)
      \end{split}
    \end{displaymath}  
    By the induction hypothesis, this is 
    \begin{displaymath}
      \sum_{\alpha_n=0}^\infty x_n^{\alpha_n} 
      \Biggl[
      \prod_{i=1}^{n-1}(1-x_1)^{-2} \sum_{B \subseteq 
        [n-1]} (-1)^{\tdeg{B}} (a_n \alpha_n +b - \sum_{j \in B}
      a_j) \prod_{j \in B} x_j 
      \Biggr] 
    \end{displaymath}
    which we write as 
    \begin{displaymath}
       \prod_{i=1}^{n-1}(1-x_1)^{-2} (S_1+S_2) 
     \end{displaymath}
     with 
     \begin{displaymath}
       S_1 =         (1-x_n)^{-1}
       \sum_{B \subseteq 
         [n-1]} (-1)^{\tdeg{B}} (b - \sum_{j \in B}
       a_j) \prod_{j \in B} x_j 
     \end{displaymath}
     and
     \begin{displaymath}
       \begin{split}
       S_2&= \sum_{\alpha_n=0}^\infty x_n^{\alpha_n}
        \sum_{B \subseteq [n-1]} (-1)^{\tdeg{B}} (a_n \alpha_n )
        \prod_{j \in B} x_j  \\
        &= \sum_{B \subseteq [n-1]} (-1)^{\tdeg{B}}   a_n
        \sum_{\alpha_n=0}^\infty \alpha_nx_n^{\alpha_n} \prod_{j
          \in B} x_j \\
        &= a_n \frac{x_n}{(1-x_n)^2} \sum_{B \subseteq [n-1]} (-1)^{\tdeg{B}} 
        \prod_{j \in B} x_j  
       \end{split}
     \end{displaymath}
     So
     \begin{displaymath}
       \begin{split}
        P &= \prod_{i=1}^n (1-x_i)^{-2} \Biggl[ \left(
        (1-x_n) \sum_{B \subseteq 
          [n-1]} (-1)^{\tdeg{B}} (b - \sum_{j \in B}
        a_j) \prod_{j \in B} x_j 
        \right)
        + \\
        & \qquad
        \left(
          a_n x_n \sum_{B \subseteq [n-1]} (-1)^{\tdeg{B}} 
        \prod_{j \in B} x_j
          \right)
        \Biggr] \\
        &= \prod_{i=1}^n (1-x_i)^{-2} 
        \sum_{B \subseteq 
          [n-1]} \Biggl[ (-1)^{\tdeg{B}} \left( a_nx_n + (1-x_n)(b - \sum_{j \in
            B} a_j)\right) \prod_{j \in B} x_j \Biggr] \\
        &=       \prod_{i=1}^n(1-x_i)^{-2} \sum_{A \subseteq [n]} \Bigl[
      (-1)^{\tdeg{A}}(b-\sum_{j \in A} a_j) \prod_{j \in A} x_j \Bigr]         
       \end{split}
     \end{displaymath}
    where the last equality is obtained by considering, for any \(B
    \subset [n-1]\), the two subsets \(B, B \cup \set{n} \subseteq [n]\).
  \end{proof}
  
  We can also note that 
  \begin{equation}
    \label{eq:3}
    \sum_{\alpha \in \Nat^n} (a_1\alpha_1 +\dots + a_n \alpha_n + b)
    x^\alpha =
    \frac{\partial}{\partial t} \left( t^b \prod_{i=1}^n (1-t^{a_i}x_i)^{-1}
    \right) \Biggl \lvert_{t=1} 
  \end{equation}
  
  We  recall the
  following well-known results:  
  
  \begin{theorem}
    \begin{enumerate}[(i)]
    \item If \(\mathcal{O} \subset T([n])\) is an order ideal, then the
      complement \(I = \mathcal{O} \setminus T([n])\) is a monoid
      ideal in \(T([n])\); i.e., \(T([n]) I \subseteq
      I\). Furthermore, \(I\) has a unique minimal generating set,
      which is finite.
    \item There exists a  (not necessarily unique) partition (i.e.
      disjoint union)
      \begin{equation}
        \label{eq:Idec}
        I = \cup_{j=1}^m u_j T(S_j), \qquad u_j \in I, \,  S_j \subseteq [n]
      \end{equation}
      with exactly one of the \(S_j\)'s is equal to \([n]\).
    \end{enumerate}    
  \end{theorem}
  \begin{proof}
    The first assertion is the so-called Dickson Lemma; the second is
    a result of Riquier \cite{Riq:Decomp}; see also Janet \cite{Janet:part} and Thomas \cite{Thomas:RET}.
      \end{proof}

      \begin{remark}
        Sometimes, a similar decomposition is called a Stanley decomposition
        \cite{stanley1982linear,apel2003conjecture}
  \end{remark}

  \begin{remark}
    In order to obtain a decomposition adapted to our purpose,
    we will accept some singletons in \eqref{eq:Idec}, i.e. some
    \(S_j=\emptyset\). 
  \end{remark}

  \begin{definition}
    Let \(\mathcal{O} \subset T([n])\) be a finite order ideal. A
    partition \eqref{eq:Idec} of \(T([n])\) is called
    \emph{admissible for \(\mathcal{O}\)} if, whenever
    the support of \(\vek{\beta}\) is contained in \(S_j\) it holds
    that 
    \begin{equation}
      \label{eq:indInc}
      \ind_{\mathcal{O}}(u_j\vek{x}^{\vek{\beta}}) =
      \ind_{\mathcal{O}}(u_j) + \tdeg{\vek{\beta}}    
    \end{equation}
  \end{definition}

  \begin{lemma}\label{lemma:admiss}
    Let \(\mathcal{O} \subset T([n])\) be a finite order ideal. If the 
    partition \eqref{eq:Idec} of \(T([n])\) is admissible for
    \(\mathcal{O}\), and \(u_j=\vek{x}^{\vek{\alpha_j}}\) then
    \begin{equation}
      \label{eq:IndAdmiss}
      \Ind_{\mathcal{O}}(y_1,\dots,y_n) = \sum_{j=1}^m
      \vek{y}^{\vek{\alpha_j}} P_{\tdeg{S_j}}(\setsuchas{y_\ell}{\ell
        \in S_j};1,\dots,1;\ind_{\mathcal{O}}(u_j)
    \end{equation}
  \end{lemma}
  \begin{proof}
    Obvious.
  \end{proof}

  \begin{example}\label{ex:311cont}
    Not all decompositions \eqref{eq:Idec} are admissible  for
    \(\mathcal{O}\). Consider again the order ideal of
    Example~\ref{ex:311}. 
    We see that, because of the ``embedded'' 1 at position \((1,1)\),
    any suitable decomposition must include \(x_1x_2\) as a singleton.

    For instance, the 
    decomposition
    \begin{displaymath}
      I = x_2^3 T([2]) \cup x_1x_2^2 T({1}) \cup x_1x_2 T({1}) \cup
      x_1^3 T({1}) 
    \end{displaymath}
    does not yield an expression
    \begin{multline*}
      \Ind_{\mathcal{O}}(y_1,y_1) = y_2^3 P_2(y_1,y_2;1,1;1) + \\
      y_1y_2^2 P_1(y_1;1;1)  + y_1y_2 P_1(y_1;1;1) +
      y_1^3 P_1(y_1;1;1)
    \end{multline*}
    since the restriction \(\ind_{\mathcal{O}}(i,1)\), \(i \ge 1\) is
    not strictly increasing.

    The decomposition used in Theorem~\ref{thm:IndDim2} is
    \begin{multline*}
      I = x_2^3 T(\set{2}) \cup x_1x_2^3 T(\set{2}) \cup x_1^2x_2^3 T(\set{2}) \cup
      x_1^3x_2^3 T(\set{1,2}) 
      \cup  \\
      x_1^3x_2^2 T(\set{1}) \cup 
      x_1^3x_2 T(\set{1}) \cup 
      x_1^3 T(\set{1}) 
    \end{multline*}
    
    A more ``economic'' decomposition is 
    \begin{multline*}
      I =  x_2^3 T(\set{1,2}) \cup x_1x_2^2 T(\set{1}) \cup
      x_1x_2T(\emptyset) \cup x_1^2x_2 T(\set{x_1}) \cup x_1^3T(\set{x_1})
    \end{multline*}
    yielding the correct expression
    \begin{multline*}
      \Ind_{\mathcal{O}}(y_1,y_2) = y_2^3P_2(y_1,y_2;1,1;1) + \\
      y_1y_2^2P_1(y_1;1;1) + y_1y_2 + y_1^2y_2P_1(y_1;1;1) + y_1^3P_1(y_1;1;1)
    \end{multline*}
  \end{example}

  \begin{definition}
    If \(\mathcal{O} \subset T([n])\) then the minimal axis-parallel
    box \(B\) containing \(\mathcal{O}\) is called the \emph{bounding box}
    of \(\mathcal{O}\). The \emph{extreme corner} of \(B\) is the
    point \((m_1,\dots,m_n)\) with \(m_i =
    \max(\setsuchas{\alpha_i}{\alpha \in \mathcal{O}}\).
  \end{definition}
  
  \begin{lemma}\label{lemma:increaseGeneral}
    Let \(\mathcal{O} \subset T([n])\) be a finite order ideal, with a
    bounding box \(B\) with extreme corner \((m_1,\dots,m_n)\). Then,
    if \(\emptyset \neq S \subseteq [n]\), if \(\alpha_i > m_i\) for \(i
    \in S\), and if \(\beta_i =0\) for \(i \not \in S\), then
    \begin{equation}
      \label{eq:10}
      \ind_{\mathcal{O}}(\alpha+\beta) = \ind_{\mathcal{O}}(\alpha) + \tdeg{\beta}
    \end{equation}
  \end{lemma}
  \begin{proof}
    If \(n=1\) the assertion is trivial. The case \(n=2\) is
    Lemma~\ref{lemma:increase2}. 
    
    For a general \(n\), 
    we will prove that if \(\alpha_1 > m_1\)
    then
    \begin{equation}
      \label{eq:specN}
      \ind_{\mathcal{O}}(\alpha + r\vek{e}_1) =
      \ind_{\mathcal{O}}(\alpha) +r
    \end{equation}
    As in Lemma~\ref{lemma:increase2}, this is enough to show the full
    assertion. Furthermore, it is enough to show this for \(r=1\).

    Let \(I=T([n]) \setminus \mathcal{O}\), and
    put
    \begin{displaymath}
      \begin{split}
        I_k & =\setsuchas{\vek{x}^{\vek{\beta}} \in I}{\beta_n=k} \\
        \mathcal{O}_k &= \setsuchas{\vek{x}^{\vek{\beta}} \in
          \mathcal{O}}{\beta_n=k} 
      \end{split}
    \end{displaymath}

    If \(k=0\), then \(\mathcal{O}_0\) can be regarded as a finite
    order ideal in \(T([n-1])\), and so by induction (on \(n\)) \eqref{eq:specN}
    holds.
    
    We now prove the assertion by induction on \(k\).
    For \(k>0\), we have that
    \begin{multline*}
      \ind_{\mathcal{O}}(m_1+1,\alpha_2,\dots,\alpha_{n-1},k) 
      = 1 + \\
      \min(
      \ind_{\mathcal{O}}(m_1,\alpha_2,\dots,\alpha_{n-1},k),
      \ind_{\mathcal{O}}(m_1+1,\alpha_2,\dots,\alpha_{n-1},k-1)
      ).
    \end{multline*}
    If
    \begin{displaymath}
      \ind_{\mathcal{O}}(m_1+1,\alpha_2,\dots,\alpha_{n-1},k) 
      = 1 + \ind_{\mathcal{O}}(m_1,\alpha_2,\dots,\alpha_{n-1},k)
    \end{displaymath}
    then we are done.
    If
    \begin{displaymath}
       \ind_{\mathcal{O}}(m_1+1,\alpha_2,\dots,\alpha_{n-1},k) 
      = 1 + \ind_{\mathcal{O}}(m_1+1,\alpha_2,\dots,\alpha_{n-1},k-1)
    \end{displaymath}
    then we note that by the induction hypothesis
    \begin{displaymath}
      \ind_{\mathcal{O}}(m_1+1,\alpha_2,\dots,\alpha_{n-1},k-1) =
      1+ \ind_{\mathcal{O}}(m_1,\alpha_2,\dots,\alpha_{n-1},k-1)
    \end{displaymath}
    hence
    \begin{displaymath}
      \ind_{\mathcal{O}}(m_1+1,\alpha_2,\dots,\alpha_{n-1},k) 
      = 2 + \ind_{\mathcal{O}}(m_1,\alpha_2,\dots,\alpha_{n-1},k-1)
    \end{displaymath}
    Using Lemma~\ref{lemma:Kehrein} we get that
    \begin{displaymath}
      \ind_{\mathcal{O}}(m_1+1,\alpha_2,\dots,\alpha_{n-1},k) 
      = 1 + \ind_{\mathcal{O}}(m_1,\alpha_2,\dots,\alpha_{n-1},k).
    \end{displaymath}
    Remark~\ref{rem:man} applies here, as well.
  \end{proof}

  \begin{definition}
    If \(C \subset T([n])\) is a finite order ideal, and if
    \(\vek{u} \in C\), we define the set of \emph{free
      directions} at \(u\) by
    \begin{equation}
      \label{eq:7}
      \mathrm{free}_{C}(\vek{u}) \setsuchas{i \in [n]}{\vek{u} +
        \vek{e}_i \not \in C}
    \end{equation}
    Here, \(\vek{e}_i \in \Nat^n\) is the vector with zeroes in all
    positions except in the \(i\)'th position, where there  is a one.
  \end{definition}

  \begin{theorem}\label{thm:IndGenForm}
    Let \(\mathcal{O} \subset T([n])\) be a finite order ideal with
    bounding box \(B\). For \(\vek{u} \in \overline{\partial B}\), put
    \(f(\vek{u}) = \mathrm{free}_{\overline{\partial B}}(\vek{u})\).
    Then the following identity holds:
    \begin{equation}
      \label{eq:IndGeneral}
      \Ind_{\mathcal{O}}(y_1,\dots,y_n) = 
      \sum_{\vek{u} \in \overline{\partial B}} 
      P_{\# f(\vek{u})}(\setsuchas{y_i}{i \in  f(\vek{u})};
        1,\dots,1; \ind_{\mathcal{O}}(\vek{u})) 
    \end{equation}   
    In particular,
    \begin{displaymath}
      \Ind_{\mathcal{O}}(y_1,\dots,y_n) \prod_{i=1}^n (1-y_i)^2
    \end{displaymath}
    is a polynomial.
  \end{theorem}
  \begin{proof}
    Follows from Lemma~\ref{lemma:increaseGeneral} and Lemma~\ref{lemma:admiss}.
  \end{proof}

  \begin{remark}
  Note that
  \begin{enumerate}
  \item  Theorem~\ref{thm:IndDim2} is a special case of
    Theorem~\ref{thm:IndGenForm}. Points within the bounding box have no
    free directions, hence contribute with a single term to
    \(\Ind_{\mathcal{O}}\). 
  \item We can not replace \(B\)
    with \(\mathcal{O}\) in \eqref{eq:IndGeneral}, since there can be
    several \(\vek{u} \in \partial \mathcal{O}\) with
    \(\mathrm{free}_{\mathcal{O}}(\vek{u}) = [n]\), leading to
    double-counting.
  \item There are other admissible decompositions of
    \(T([n])\setminus \mathcal{O}\), but as
    example~\ref{ex:311} shows, we might have to include singleton sets in
    the decomposition.
  \end{enumerate}
\end{remark}

\end{subsection}
\end{section}

\raggedright
\bibliography{border}
\bibliographystyle{amsplain}
\end{document}